\documentclass{amsart}
\usepackage{amsmath}
\usepackage{amsthm}
\usepackage{amsfonts}
\usepackage{amssymb}
\usepackage{mathrsfs}
\usepackage{amscd}
\usepackage{url}

\newcommand{\re}{\mathbb{R}}
\newcommand{\cpx}{\mathbb{C}}

\newcommand{\Z}{\mathbb{Z}}

\newcommand{\N}{\mathbb{N}}

\newcommand{\diag}{\mbox{diag}}

\newcommand{\half}{\frac{1}{2}}
\newcommand{\lmd}{\lambda}

\newcommand{\eps}{\epsilon}

\def\rank{\mbox{rank}}

\newcommand{\sig}{\sigma}
\newcommand{\Sig}{\Sigma}

\newcommand{\reff}[1]{(\ref{#1})}
\newcommand{\pt}{\partial}
\newcommand{\prm}{\prime}

\newcommand{\mc}[1]{\mathcal{#1}}

\newcommand{\bdes}{\begin{description}}
\newcommand{\edes}{\end{description}}
\newcommand{\bal}{\begin{align}}
\newcommand{\eal}{\end{align}}
\newcommand{\bnum}{\begin{enumerate}}
\newcommand{\enum}{\end{enumerate}}
\newcommand{\bit}{\begin{itemize}}
\newcommand{\eit}{\end{itemize}}
\newcommand{\bea}{\begin{eqnarray}}
\newcommand{\eea}{\end{eqnarray}}
\newcommand{\be}{\begin{equation}}
\newcommand{\ee}{\end{equation}}
\newcommand{\baray}{\begin{array}}
\newcommand{\earay}{\end{array}}
\newcommand{\bsry}{\begin{subarray}}
\newcommand{\esry}{\end{subarray}}
\newcommand{\bca}{\begin{cases}}
\newcommand{\eca}{\end{cases}}
\newcommand{\bcen}{\begin{center}}
\newcommand{\ecen}{\end{center}}
\newcommand{\bbm}{\begin{bmatrix}}
\newcommand{\ebm}{\end{bmatrix}}
\newcommand{\bmx}{\begin{matrix}}
\newcommand{\emx}{\end{matrix}}
\newcommand{\bpm}{\begin{pmatrix}}
\newcommand{\epm}{\end{pmatrix}}
\newcommand{\btab}{\begin{tabular}}
\newcommand{\etab}{\end{tabular}}

\newtheorem{theorem}{Theorem}[section]

\newtheorem{prop}[theorem]{Proposition}

\newtheorem{lemma}[theorem]{Lemma}

\newtheorem{cond}[theorem]{Condition}
\theoremstyle{definition}

\newtheorem{exm}[theorem]{Example}

\begin{document}

\title[Optimality Conditions and Convergence of Lasserre's Hierarchy]
{Optimality Conditions and Finite Convergence of Lasserre's Hierarchy}

\author[Jiawang Nie]{Jiawang Nie}
\address{Department of Mathematics\\
  University of California \\
  San Diego}
\email{njw@math.ucsd.edu}

\begin{abstract}
Lasserre's hierarchy is a sequence of semidefinite relaxations
for solving polynomial optimization problems globally.
This paper studies the relationship between optimality conditions
in nonlinear programming theory
and finite convergence of Lasserre's hierarchy.
Our main results are:
i) Lasserre's hierarchy has finite convergence
when the constraint qualification, strict complementarity
and second order sufficiency conditions hold
at every global minimizer, under the standard archimedean condition;
the proof uses a result of Marshall on boundary hessian conditions.
ii) These optimality conditions are all satisfied at every local minimizer
if a finite set of polynomials, which are in the coefficients of input polynomials,
do not vanish at the input data (i.e., they hold in a Zariski open set).
This implies that, under archimedeanness,
Lasserre's hierarchy has finite convergence generically.
\end{abstract}



\maketitle

\section{Introduction}

Given polynomials $f, h_i, g_j$ in $x\in\re^n$,
consider the optimization problem
\be  \label{pop:gen}
\left\{\baray{rl}
\min & f(x) \\
s.t. & h_i(x) = 0 \,(i=1,\ldots,m_1),  \\
 & g_j(x) \geq  0 \,(j=1,\ldots,m_2).
\earay \right.
\ee
Let $K$ be the feasible set of \reff{pop:gen}.
When $m_1=0$ (resp. $m_2=0$), there are no equality (resp. inequality) constraints.
For convenience, denote $h:=(h_1,\ldots,h_{m_1})$,
$g:=(g_1,\ldots,g_{m_2})$ and $g_0:=1$.
A standard approach for solving \reff{pop:gen} globally is
{\it Lasserre's hierarchy} of semidefinite programming (SDP) relaxations \cite{Las01}.
It is based on a sequence of SOS type representations
for polynomials that are nonnegative on $K$.
To describe Lasserre's hierarchy, we first introduce some notation.
Let $\re[x]$ be the ring of polynomials with real coefficients
and in $x:=(x_1,\ldots,x_n)$.
A polynomial $p \in \re[x]$ is said to be SOS
if $p=p_1^2+\cdots+p_k^2$ for $p_1,\ldots, p_k \in \re[x]$.
The set of all SOS polynomials is denoted by $\Sig \re[x]^2$.
For each $k \in \N$ ($\N$ is the set of nonnegative integers), denote
\[
\langle h \rangle_{2k}: =
\left\{
\left. \overset{m_1}{ \underset{ i=1}{\sum} }  \phi_i h_i \right|
\baray{c}
\mbox{ each } \phi_i \in \re[x] \\
\mbox{ and } \deg (\phi_i h_i ) \leq 2k
\earay
\right\},
\]
\[
Q_k(g):= \left\{
\left. \overset{m_2}{ \underset{ j=0}{\sum} }   \sig_j g_j \right|
\baray{c}
\mbox{each } \sig_j \in  \Sig\re[x]^2  \\
\mbox{ and } \deg( \sig_j g_j) \leq 2k
\earay
\right\}.
\]
The set $\langle h \rangle_{2k}$ is called
the $2k$-th {\it truncated ideal} generated by $h$,
and $Q_k(g)$ is called the $k$-th {\it truncated quadratic module}
generated by $g$. Lasserre's hierarchy is the sequence
of SOS relaxations ($k\in\N$ is called a relaxation order):
%
%
\be  \label{sos:Put}
\max \quad \gamma \quad \mbox{s.t.} \quad
f - \gamma  \in  \langle h \rangle_{2k} + Q_k(g).
\ee
%
%
The SOS program \reff{sos:Put} is equivalent to a semidefinite program \cite{Las01}.
We refer to \cite{LasBok,Lau} for surveys in this area.

Let $f_{min}$ denote the minimum value of \reff{pop:gen} and
$f_k$ denote the optimal value of \reff{sos:Put}. Clearly, $f_k\leq f_{min}$ for all $k$
and $\{f_k\}$ is monotonically increasing. Under the archimedean condition
(i.e., $R- \Sig_{i=1}^n x_i^2 \in \langle h \rangle_{2t}+ Q_t(g)$
for some $t \in \N$ and $R>0$),
Lasserre obtained the asymptotic convergence $f_k \to f_{min}$ as $k\to \infty$,
by using Putinar's Positivstellensatz (cf.~Theorem~\ref{Put-Pos}).
When $f_k = f_{min}$ for some $k$,
we say Lasserre's hierarchy has {\it finite convergence}.
When $h(x)=0$ defines a finite set in the complex space $\cpx^n$,
Laurent \cite{Lau07} proved that Lasserre's hierarchy has finite convergence.
%
%
Indeed, when $h(x)=0$ defines a finite set in $\re^n$,
the sequence $\{f_k\}$ also has finite convergence to $f_{min}$,
as shown in \cite{Nie-PopVar}.
There exist examples that Lasserre's hierarchy fails to have finite convergence,
e.g., when $f$ is the Motzkin polynomial
$x_1^2x_2^2(x_1^2+x_2^2-3x_3^2)+x_3^6$
and $K$ is the unit ball \cite[Example~5.3]{Nie-jac}.
Indeed, such examples always exist
when $\dim(K)\geq 3$ (cf.~Scheiderer \cite[Prop.~6.1]{Sch99}).

However, in practical applications,
Lasserre's hierarchy often has finite convergence,
e.g., as shown by numerical experiments in
Henrion and Lasserre~\cite{HenLas03,HenLas05}.
The known examples for which finite convergence fails
are created in very special ways.
%
%
Since Lasserre proposed his method in \cite{Las01},
people are intrigued very much by the discrepancy between its theory
(only asymptotic convergence is guaranteed theoretically) and its practical performance
(in applications we often observe finite convergence).
The motivation of this paper is trying to resolve this discrepancy.
Our main result is that Lasserre's hierarchy has finite convergence
when a finite set of polynomials, which are in the coefficients
of $f$ and all $h_i,g_j$,
do not vanish at the input data, under the archimedean condition.
This implies that, under archimedeanness,
Lasserre's hierarchy has finite convergence generically.
(We say a property holds {\it generically}
if it holds in the entire space of input data
except a set of Lebsgue measure zero.)
To prove this, we need to investigate optimality conditions for \reff{pop:gen}.

We here give a short review of optimality conditions
in nonlinear programming theory (cf. \cite[Section~3.3]{Brks}).
Let $u$ be a local minimizer of \reff{pop:gen} and
$J(u)=\{j_1, \ldots, j_r\}$ be the index set of active inequality constraints.
If the {\it constraint qualification condition (CQC)} holds at $u$, i.e., the gradients
\[
\nabla h_1(u), \ldots, \nabla h_{m_1}(u),
\nabla g_{m_1}(u), \ldots, \nabla g_{j_r}(u)
\]
are linearly independent, then
there exist Lagrange multipliers $\lmd_1,\ldots,\lmd_{m_1}$
and $\mu_{1},\ldots,\mu_{m_2} $ satisfying
\be  \label{opcd:grad}
\nabla f(u) =  \sum_{i=1}^{m_1} \lmd_i \nabla h_i(u) +
\sum_{j =1 }^{m_2} \mu_j \nabla g_j (u),
\ee
\be \label{op-cd:comc}
\mu_{1} g_{1}(u) = \cdots = \mu_{m_2} g_{m_2}(u) = 0, \quad
\mu_{1} \geq 0, \ldots,  \mu_{m_2} \geq 0.
\ee
The equation \reff{opcd:grad} is called the {\it first order optimality condition (FOOC)},
and \reff{op-cd:comc} is called the {\it complementarity condition}.
If it further holds that
\be \label{optcd-scc}
\mu_{1} + g_{1}(u) >0, \ldots,  \mu_{m_2} + g_{m_2}(u) > 0,
\ee
we say the {\it strict complementarity condition (SCC)} holds at $u$.
Note that strict complementarity is equivalent to $\mu_j >0$ for every $j \in J(u)$.
Let $L(x)$ be the associated Lagrange function
\[
L(x) := f(x) - \sum_{i=1}^{m_1} \lmd_i h_i(x) -
\sum_{j\in J(u)} \mu_j  g_j (x).
\]
Clearly, \reff{opcd:grad} implies $\nabla_x L(u) = 0$.
The polynomials $f,h_i,g_j$ are infinitely many times differentiable everywhere.
Thus, under the constraint qualification condition,
the {\it second order necessity condition (SONC)} holds at $u$, i.e.,
\be \label{opt:sonc}
v^T \nabla_x^2 L(u) v \geq 0 \quad  \mbox{ for all } \, v \in G(u)^\perp.
\ee
Here, $G(x)$ denotes the Jacobian of the active constraining polynomials
\[
G(x) = \bbm \nabla h_1(x) & \cdots & \nabla h_{m_1}(x) &
\nabla g_{j_1}(x) & \cdots & \nabla g_{j_r}(x) \ebm^T
\]
and $G(u)^\perp$ denotes the null space of $G(u)$. If it holds that
\be \label{sosc-optcnd}
v^T \nabla_x^2 L(u) v > 0 \quad  \mbox{ for all } \, 0 \ne v \in G(u)^\perp,
\ee
we say the {\it second order sufficiency condition (SOSC)} holds at $u$.

We summarize the above as follows.
If the constraint qualification condition holds at $u$,
then \reff{opcd:grad}, \reff{op-cd:comc} and \reff{opt:sonc} are necessary conditions for
$u$ to be a local minimizer of $f$ on $K$, but they are not sufficient.
If \reff{opcd:grad}, \reff{op-cd:comc}, \reff{optcd-scc} and \reff{sosc-optcnd}
hold at a point $u \in K$, then $u$ is a strict local minimizer of \reff{pop:gen}.
The first order optimality, strict complementarity and
second order sufficiency conditions are sufficient for strict local optimality.
We refer to \cite[Section~3.3]{Brks}.

This paper studies the relationship between
optimality conditions and finite convergence of Lasserre's hierarchy.
Denote $\re[x]_d :=\{ p\in \re[x]: \deg(p) \leq d\}$
and $[m]:=\{1,\ldots,m\}$.
Our main conclusions are the following two theorems.

\begin{theorem} \label{mthm:opc=>fcvg}
Suppose the archimedean condition holds
for the polynomial tuples $h$ and $g$ in \reff{pop:gen}.
If the constraint qualification, strict complementarity and
second order sufficiency conditions
hold at every global minimizer of \reff{pop:gen},
then Lasserre's hierarchy of \reff{sos:Put} has finite convergence.
\end{theorem}

\begin{theorem} \label{mthm:opcd=gen}
Let $d_0, d_1, \ldots, d_{m_1}, d_1^{\prm}, \ldots, d_{m_2}^{\prm}$ be positive integers.
Then there exist a finite set of polynomials $\varphi_1,\ldots,\varphi_L$
(cf.~Condition~\ref{cond:gen}),
which are in the coefficients of polynomials
$f \in \re[x]_{d_0}$, $h_i \in \re[x]_{d_i}$ ($i\in [m_1]$),
$g_j \in \re[x]_{d_j^{\prm}}$ ($j\in [m_2]$),
such that if $\varphi_1,\ldots,\varphi_L$ do not vanish at the input polynomial,
then the constraint qualification, strict complementarity and
second order sufficiency conditions
hold at every local minimizer of \reff{pop:gen}.
\end{theorem}

The proof of Theorem~\ref{mthm:opc=>fcvg} uses
a result of Marshall on boundary hessian conditions \cite{Mar06,Mar09},
and the proof of Theorem~\ref{mthm:opcd=gen}
uses elimination theory in computational algebra.
Theorem~\ref{mthm:opcd=gen} implies that
these classical optimality conditions hold in a Zariski open set
in the space of input polynomials with given degrees.
The paper is organized as follows.
Section~2 presents some backgrounds in the field;
Section~3 is mostly to prove Theorem~\ref{mthm:opc=>fcvg};
Section~4 is mostly to prove Theorem~\ref{mthm:opcd=gen};
Section~5 makes some discussions.

\section{Preliminary}
\setcounter{equation}{0}

\subsection{Notation}
The symbol $\re$ (resp., $\cpx$) denotes the set of real (resp., complex) numbers.
A polynomial is called a form if it is homogeneous.
For $f \in \re[x]$, $\widetilde{f}$ denotes the homogenization of $f$,
i.e., $\widetilde{f}(\tilde{x}) = x_0^{\deg(f)} \cdot f(x/x_0)$
with $\tilde{x}:=(x_0,x_1,\ldots,x_n)$.
The symbol $\|\cdot\|_2$ denotes the standard $2$-norm.
For a symmetric matrix $X$, $X\succeq 0$ (resp., $X\succ 0$) means
$X$ is positive semidefinite (resp. positive definite).
The determinant of a square matrix $A$ is $\det A$.
The $N\times N$ identity matrix is denoted as $I_N$.
If $p$ is a polynomial in $x$, $\nabla p$ (resp., $\nabla^2 p$) denotes the
gradient (resp., Hessian) of $p$ with respect to $x$;
if $p$ has variables in addition to $x$, $\nabla_x p$ (resp., $\nabla_x^2 p$)
denotes the gradient (resp., Hessian) of $p$ with respect to $x$.
For $p_1,\ldots,p_r \in \re[x]$, $Jac(p_1,\ldots,p_r)|_u$ denotes
the Jacobian of $(p_1,\ldots, p_r)$ at $u$, i.e.,
$Jac(p_1,\ldots,p_r)|_u = (\pt p_i(u)/\pt x_j)_{1\leq i\leq r, 1\leq j\leq n}$.

\subsection{Some basics in real algebra}

Here we give a short review on elementary real algebra.
More details can be found in \cite{BCR,CLO97}.

An ideal $I$ of $\re[x]$ is a subset such that $ I \cdot \re[x] \subseteq I$
and $I+I \subseteq I$. Given $p_1,\ldots,p_m \in \re[x]$,
$\langle p_1,\cdots,p_m \rangle $ denotes
the smallest ideal containing all $p_i$, which is the set
$p_1 \cdot \re[x] + \cdots + p_m  \cdot \re[x]$.
A variety is a subset of $\cpx^n$ that consists of
common zeros of a set of polynomials.
A real variety is the intersection of a variety and the real space $\re^n$.
Given a polynomial tuple $p=(p_1,\ldots, p_r)$, denote
\begin{align*}
V(p) & := \{v\in \cpx^n: \, p_1(v) = \cdots = p_r(v) = 0 \}, \\
V_{\re}(p) & := \{v\in \re^n: \, p_1(v) = \cdots = p_r(v) = 0 \}.
\end{align*}
Every set $T \subset \re^n$ is contained in a real variety.
The smallest one containing $T$
is called the {\it Zariski} closure of $T$, and is denoted by $Zar(T)$.
In the Zariski topology on $\re^n$, the real varieties are closed sets,
and the complements of real varieties are open sets.
Denote $I(T):=\{q\in \re[x]: \, q(u) = 0 \,\forall \, u \in T\}$,
which is an ideal in $\re[x]$ and
is called the {\it vanishing ideal} of $T$.

Let $h=(h_1,\ldots,h_{m_1})$ and $g=(g_1,\ldots,g_{m_2})$
be the polynomial tuples as in \reff{pop:gen}, and
$K$ be the feasible set of \reff{pop:gen}.
Recall the definitions of $\langle h \rangle_{2k}$ and $Q_k(g)$ in the Introduction.
Clearly, the union $\cup_{k\in \N} \langle h \rangle_{2k}$
is the ideal $\langle h \rangle := \langle h_1,\ldots,h_{m_1} \rangle$.
The union $Q(g):=\cup_{k\in\N} Q_k(g)$ is called
the {\it quadratic module} generated by $g$.
The set $\langle h \rangle + Q(g)$ is called {\it archimedean}
if $R - \|x\|_2^2 \in \langle h \rangle + Q(g)$ for some $R>0$.
Clearly, if $p \in \langle h \rangle + Q(g)$, then $p$ is nonnegative on $K$,
while the converse is not always true.
However, if $p$ is positive on $K$ and $\langle h \rangle + Q(g)$
is archimedean, then $p \in \langle h \rangle + Q(g)$.
This is called {\it Putinar's Positivstellensatz}.

\begin{theorem} [Putinar, \cite{Put}] \label{Put-Pos}
Let $K$ be the feasible set of \reff{pop:gen}.
Suppose $\langle h \rangle + Q(g)$ is archimedean.
If $p \in \re[x]$ is positive on $K$, then $p \in \langle h \rangle + Q(g)$.
\end{theorem}

\subsection{The boundary hessian condition}

Let $K$ be the feasible set of \reff{pop:gen} and $h=(h_1,\ldots,h_{m_1})$.
Let $u$ be a local minimizer of \reff{pop:gen},
and $\ell$ be the local dimension of $V_{\re}(h)$ at $u$ (cf.~\cite[\S2.8]{BCR}).
We first state a condition about parameterizing $K$ around $u$ locally,
which was proposed by Marshall.

\begin{cond}[Marshall,\cite{Mar09}] \label{cod:par-u}
i) The point $u$ on $V_{\re}(h)$ is nonsingular
and there exists a neighborhood $\mc{O}$ of $u$ such that
$V_{\re}(h) \cap \mc{O}$ is parameterized by
uniformizing parameters $t_1, \ldots, t_\ell$;
ii) there exist $1\leq \nu_1 < \cdots < \nu_r \leq m_2$, such that
$t_j = g_{\nu_j}$ ($j=1,\ldots,r$) on $V_{\re}(h) \cap \mc{O}$
and $K \cap \mc{O}$ is defined by
$t_1 \geq 0, \ldots, t_r \geq 0$.
\end{cond}

The following condition was introduced by Marshall \cite{Mar06,Mar09} in
studying Putinar type representation for nonnegative polynomials,
and it is called the {\it boundary hessian condition} (BHC).

\begin{cond}[Marshall,\cite{Mar06,Mar09}] \label{def:BDH}
Assume Condition~\ref{cod:par-u} holds.
Expand $f$ locally around $u$ as $f=f_0+f_1 +f_2 + \cdots$,
with every $f_i$ being homogeneous of degree $i$ in $t_1,\ldots,t_\ell$.
The linear form $f_1 = a_1 t_1 +\cdots + a_r t_r$
for some positive constants $a_1 >0, \ldots, a_r >0$,
and the quadratic form $f_2(0, \ldots, 0, t_{r+1}, \ldots, t_{\ell})$
is positive definite in $(t_{r+1}, \ldots, t_{\ell})$.
\end{cond}

If $K$ is compact and the boundary hessian condition holds
at every global minimizer,
then \reff{pop:gen} has finitely many global minimizers.
(See the proof of Theorem~9.5.3 in \cite{MarBk}.)
Marshall proved the following important result.

\begin{theorem}(Marshall, \cite[Theorem~9.5.3]{MarBk}) \label{thm:marBHC}
Let $V=V_{\re}(h)$ and $f_{min}$ be the minimum of \reff{pop:gen}.
If $\langle h \rangle + Q(g)$ is archimedean and
the boundary hessian condition holds at every global minimizer of \reff{pop:gen},
then $f-f_{min} \in I(V)+Q(g)$.
\end{theorem}

%
In the above, if $I(V) = \langle h \rangle$
(i.e., $\langle h \rangle$ is {\it real}, \cite[\S4.1]{BCR}),
then $f-f_{min} \in \langle h \rangle + Q(g)$.
Theorem~\ref{thm:marBHC} can also be found in
Scheiderer's survey \cite[Theorem~3.1.7]{Sch09}.

\subsection{Resultants and discriminants}

Here, we review some basics of resultants and discriminants.
We refer to \cite{CLO98,GKZ,Nie-dis,Stu02} for more details.

Let $f_1,\ldots, f_n$ be forms in $x=(x_1,\ldots,x_n)$.
The resultant $Res(f_1,\ldots, f_n)$ is a polynomial,
in the coefficients of $f_1, \ldots,f_n$, having the property that
\[
Res(f_1,\ldots,f_n) = 0 \quad  \Longleftrightarrow \quad
\exists \, 0 \ne u\in \cpx^n, \, f_1(u)=\cdots = f_n(u)=0.
\]
The discriminant of a form $f$ is defined as
\[
\Delta(f) \, := \, Res\left(\frac{\pt f}{\pt x_1},\ldots, \frac{\pt f}{\pt x_n} \right).
\]
So, it holds that
\[
\Delta(f) = 0 \quad  \Longleftrightarrow \quad
\exists \, 0 \ne u\in \cpx^n, \, \nabla f(u)=0.
\]
Both $Res(f_1,\ldots, f_n)$ and $\Delta(f)$
are homogeneous, irreducible and have integer coefficients.

Discriminants and resultants are also defined for nonhomogeneous polynomials.
If one of $f_0,f_1,\ldots, f_n$ is not a form in $x$,
then $Res(f_0,f_1,\ldots, f_n)$ is defined to be
$Res(\widetilde{f_0}, \ldots, \widetilde{f_n})$,
where each $\widetilde{f_i}$ is the homogenization of $f_i$.
Similarly, if $f$ is not a form, then
$\Delta(f)$ is defined to be $\Delta(\widetilde{f})$.

Discriminants are also defined for several polynomials \cite{Nie-dis}.
Let $f_1,\ldots, f_m$ be forms in $x$ of degrees $d_1,\ldots,d_m$ respectively,
and $m \leq n-1$. Suppose at least one $d_i>1$.
The discriminant of $f_1,\ldots,f_m$, denoted by $\Delta(f_1,\ldots,f_m)$,
is a polynomial in the coefficients of $f_1,\ldots, f_m$, having the property that
$\Delta(f_1,\ldots,f_m) = 0$
if and only if there exists $0 \ne u \in \cpx^n$ satisfying
\be \label{dfeqn:mul-dis}
f_1(u) = \cdots = f_m(u) = 0, \quad
\rank \bbm \nabla f_1(u) &  \cdots & \nabla f_m(u) \ebm \, < \, m.
\ee
If one of $f_1,\ldots,f_m$ is nonhomogeneous and $m\leq n$,
then $\Delta(f_1,\ldots,f_m)$ is defined to be
$\Delta(\widetilde{f_1}, \ldots,\widetilde{f_m})$.
In the nonhomogeneous case, $\Delta(f_1,\ldots,f_m) =0$
if there exists $u \in \cpx^n$ satisfying \reff{dfeqn:mul-dis} (cf. \cite{Nie-dis}).

We conclude this section with an elimination theorem
for general homogeneous polynomial systems.

\begin{theorem}(Elimination Theory, \cite[Theorem~5.7A]{Hrtsn}) \label{thm:elim}
Let $f_1, \ldots, f_r$ be homogeneous polynomials in $x_0,\ldots, x_n$,
having indeterminate coefficients $a_{ij}$.
Then there is a set $g_1,\ldots, g_t$ of polynomials in the $a_{ij}$,
with integer coefficients, which
are homogeneous in the coefficients of each $f_i$ separately,
with the following property: for any field $k$, and for any set of special
values of the $a_{ij} \in k$, a necessary and sufficient condition
for the $f_i$ to have a common zero different from $(0,\ldots,0)$
is that the $a_{ij}$ are a common zero of the polynomials $g_j$.
\end{theorem}

\section{Optimality conditions and Finite Convergence}
\setcounter{equation}{0}

This section is to prove Theorem~\ref{mthm:opc=>fcvg}.
It is based on the following theorem.

\begin{theorem} \label{thm:opcd=>BHC}
Let $u$ be a local minimizer of \reff{pop:gen}.
If the constraint qualification, strict complementarity and
second order sufficiency conditions hold at $u$,
then $f$ satisfies the boundary hessian condition at $u$.
\end{theorem}
\begin{proof}
Let $J(u) := \{ j_1, \ldots, j_r \}$ be the index set
of inequality constraints that are active at $u$.
For convenience, we can generally assume $u=0$, up to a shifting.
Since the constraint qualification condition holds at $0$, the gradients
\[
\nabla h_1(0), \ldots,  \nabla h_{m_1}(0),
\nabla g_{j_1} (0), \ldots, \nabla g_{j_r}(0)
\]
are linearly independent.
The origin $0$ is a nonsingular point of the real variety $V_{\re}(h)$,
because the gradients $\nabla h_1(0), \ldots,  \nabla h_{m_1}(0)$
are linearly independent.
Up to a linear coordinate transformation,
we can further assume that
\be \label{nab:gh=unit}
\left\{\baray{rcl}
\bbm \nabla g_{j_1} (0) & \cdots &  \nabla g_{j_r} (0) \ebm
& = & \bbm I_r \\ 0 \ebm, \\
\bbm \nabla h_1 (0) & \cdots &  \nabla h_{m_1} (0) \ebm
& = & \bbm 0 \\ I_{m_1} \ebm.
\earay\right.
\ee
Let $\ell := n - m_1$, which is the local dimension of $V_{\re}(h_1, \ldots, h_{m_1})$
at $0$ (cf.~\cite[Prop.~3.3.10]{BCR}).
Define a function $\varphi(x):=
(\varphi_{I}(x), \, \varphi_{II}(x) ,\, \varphi_{III}(x)): \re^n \to \re^n $ as
\be  \label{df:varphi}
\varphi_{I}(x) = \bbm g_{j_1} (x) \\ \vdots \\ g_{j_r}(x) \ebm, \quad
\varphi_{II}(x) = \bbm x_{r+1} \\ \vdots \\  x_{\ell} \ebm, \quad
\varphi_{III}(x) = \bbm h_1(x) \\ \vdots \\ h_{m_1}(x) ) \ebm.
\ee
Clearly, $\varphi(0)=0$, and the Jacobian of $\varphi$ at $0$
is the identity matrix $I_n$.
Thus, by the implicit function theorem, in a neighborhood $\mc{O}$ of $0$,
the equation $t=\varphi(x)$ defines a smooth function $x = \varphi^{-1}(t)$.
So, $t=(t_1,\ldots, t_n)$ can serve as a coordinate system
for $\re^n$ around $0$ and $t=\varphi(x)$.
In the $t$-coordinate system and in the neighborhood $\mc{O}$,
$V_{\re}(h_1,\ldots,h_{m_1})$ is defined by linear equations $t_{\ell+1}=\cdots=t_{n}=0$,
and $K \cap \mc{O}$ can be equivalently described as
\[
t_1 \geq 0, \ldots, t_r \geq 0, \quad
t_{\ell+1} = \cdots =t_n=0.
\]
Let $\lmd_i (i \in [m_1])$ and $\mu_j (j \in [m_2])$ be the Lagrange multipliers
satisfying \reff{opcd:grad}-\reff{op-cd:comc}. Define the Lagrange function
\[
L(x) := f(x) - \sum_{i=1}^{m_1} \lmd_i h_i(x) - \sum_{k=1}^{r} \mu_{j_k} g_{j_k}(x).
\]
Note that $\nabla_x L(0)=0$. In the $t$-coordinate system, define functions
\[
F(t) :=  f(\varphi^{-1}(t)), \quad
\widehat{L}(t) := L(\varphi^{-1}(t)) =
F(t)- \sum_{i=\ell+1}^{n} \lmd_{i-\ell} t_i - \sum_{k=1}^{r} \mu_{j_k} t_k.
\]
Clearly, $\nabla_{x} L(0)=0$ implies $\nabla_t \widehat{L}(0)=0$.
So, it holds that
\[
\frac{\pt F(0)}{\pt t_k} = \mu_{j_k} \, (k=1,\ldots, r),
\]
\[
\frac{\pt F(0)}{\pt t_k} = 0 \, (k=r+1,\ldots,\ell),
\]
\[
\frac{\pt F(0)}{\pt t_k} = \lmd_{k-\ell} \, (k=\ell+1,\ldots,n).
\]
Expand $F(t)$ locally around $0$ as
\[
F(t) = f_0 + f_1(t) + f_2(t) + f_3(t) + \cdots
\]
where each $f_i$ is a form in $t$ of degree $i$. Clearly, we have
\[
f_1(t)  = \mu_{j_1} t_1 + \cdots + \mu_{j_r} t_r
\quad \mbox{ on } \quad  t_{\ell+1}=\cdots=t_n=0.
\]
For $t_{r+1},\ldots, t_{\ell}$ near zero, it holds that
\[
F(0, \ldots, 0, t_{r+1},\ldots, t_{\ell}, 0, \ldots, 0) =
\widehat{L}(0, \ldots, 0, t_{r+1},\ldots, t_{\ell}, 0, \ldots, 0)=
\]
\[
L\big(\varphi^{-1}(0, \ldots, 0, t_{r+1},\ldots, t_{\ell}, 0, \ldots, 0)\big).
\]
Denote $x(t):=\varphi^{-1}(t) = (\varphi_1^{-1}(t),\ldots, \varphi_n^{-1}(t))$.
For all $i,j$, we have
\[
\frac{\pt^2 \widehat{L}(t)}{\pt t_i \pt t_j} =
\sum_{1\leq k,s\leq n} \frac{\pt^2 L( x(t) )}{\pt x_k \pt x_s}
\frac{\pt \varphi^{-1}_k(t)}{\pt t_i} \frac{\pt \varphi^{-1}_s(t)}{\pt t_j} +
\sum_{1\leq k \leq n} \frac{\pt L( x(t) )}{\pt x_k}
\frac{\pt^2 \varphi^{-1}_k(t)}{\pt t_i \pt t_j}.
\]
Evaluating the above at $x=t=0$, we get (note $\nabla_x L(0)=0$)
\[
\frac{\pt^2 \widehat{L}(0)}{\pt t_i \pt t_j} =
\sum_{1\leq k,s\leq n} \frac{\pt^2 L( 0 )}{\pt x_k \pt x_s}
\frac{\pt \varphi^{-1}_k(0)}{\pt t_i} \frac{\pt \varphi^{-1}_s(0)}{\pt t_j} .
\]
Note that $Jac(\varphi)|_0=Jac(\varphi^{-1})|_0=I_n$.
So, for all $r+1\leq i,j \leq \ell$, we have
\be \label{2nd-deriv=}
\left.\frac{\pt^2 f_2}{\pt t_i \pt t_j}\right|_{t=0} =
\left.\frac{\pt^2 F}{\pt t_i \pt t_j}\right|_{t=0} =
\left.\frac{\pt^2 \widehat{L}}{\pt t_i \pt t_j}\right|_{t=0} =
\left.\frac{\pt^2 L}{\pt x_i \pt x_j}\right|_{x=0}.
\ee

The strict complementarity condition \reff{optcd-scc}
implies that $\mu_{j_1}>0, \ldots, \mu_{j_r}>0$. So, the coefficients of
the linear form $\mu_{j_1} t_1 + \cdots + \mu_{j_r} t_r$ are all positive.
The second order sufficiency condition \reff{sosc-optcnd} implies that the sub-Hessian
\[
\left(\frac{\pt^2 L(0)}{\pt x_i \pt x_j}\right)_{r+1\leq i,j \leq \ell}
\]
is positive definite. By \reff{2nd-deriv=}, the quadratic form
$f_2$ is positive definite in $(t_{r+1},\ldots, t_{\ell})$.
Therefore, $f$ satisfies the boundary hessian condition at $0$.
\end{proof}

Now, we give the proof of Theorem~\ref{mthm:opc=>fcvg}.

\begin{proof}[Proof of Theorem~\ref{mthm:opc=>fcvg}]
By Theorem~\ref{thm:opcd=>BHC}, we know the boundary hessian condition
is satisfied at every global minimizer of $f$ on $K$,
when the constraint qualification, strict complementarity
and second order sufficiency conditions hold.
Then, by Theorem~\ref{thm:marBHC} of Marshall,
we know there exists $\sig_1 \in Q(g)$ such that
\[
f-f_{min} \equiv \sig_1 \quad  \mod \quad  I(V_{\re}(h)).
\]
Let $\hat{f}:=f-f_{min} -\sig_1$. Then $\hat{f}$ vanishes identically on $V_{\re}(h)$.
By Real Nullstellensatz (cf.~\cite[Corollary~4.1.8]{BCR}),
there exist $\ell \in \N$ and $\sig_2 \in \Sig\re[x]^2$ such that
\[
\hat{f}^{2\ell} + \sig_2 \in \langle h \rangle.
\]
Let $c>0$ be big enough such that
$s(t):=1+t+c t^{2\ell}$ is an SOS univariate polynomial in $t$
(cf.~\cite[Lemma~2.1]{Nie-PopVar}).
For each $\eps>0$, let
\[
\sig_\eps := \eps s\left( \hat{f}/\eps \right) + c \eps^{1-2\ell} \sig_2 + \sig_1.
\]
Then, one can verify that
\[
\phi_\eps := f - (f_{\min} -\eps) - \sig_\eps =
-c \eps^{1-2\ell} (\hat{f}^{2\ell} + \sig_2 ) \in \langle h \rangle.
\]
Clearly, there exists $k_0\in\N$ such that
$\sig_\eps \in Q_{k_0}(g)$ and $\phi_\eps \in \langle h \rangle_{2k_0}$
for all $\eps >0$. So, for every $\eps>0$, $\gamma = f_{min}-\eps$
is feasible in \reff{sos:Put} for the order $k_0$.
Hence, $ f_{k_0} \geq f_{min} -\eps$.
Since $\eps>0$ can be arbitrary, we get $f_{k_0} \geq f_{min}$.
Recall that $f_k \leq f_{min}$ for all $k$
and $\{f_k\}$ is monotonically increasing.
Hence, we get $f_k = f_{min}$ for all $k\geq k_0$, i.e.,
Lasserre's hierarchy has finite convergence.
\end{proof}

Theorem~\ref{thm:opcd=>BHC} shows that
the constraint qualification, strict complementarity
and second order sufficiency conditions
imply the boundary hessian condition.
Typically, to check the boundary hessian condition by its definition,
one needs to construct a local parametrization for the feasible set $K$
and verify some sign conditions,
which would be very inconvenient in applications.
However, checking optimality conditions is generally much more convenient,
because it does not need a parametrization and only requires
some elementary linear algebra operations.
This is an advantage of optimality conditions
over the boundary hessian condition.
We show this in the following example.

\begin{exm} \label{exm:opt-appl}
Consider the optimization problem:
\[
\left\{\baray{rl}
\min & x_1^6+x_2^6+x_3^6+3x_1^2x_2^2x_3^2-x_1^4(x_2^2+x_3^2)-x_2^4(x_3^2+x_1^2)-x_3^4(x_1^2+x_2^2) \\
s.t. &  x_1^2+x_2^2+x_3^2 = 1.
\earay\right.
\]
The objective is the Robinson form which is nonnegative but not SOS (cf. \cite{Rez00}).
The minimum $f_{min}=0$, and the global minimizers are
\[
\frac {1}{\sqrt{3}}(\pm 1, \pm 1, \pm 1),
\frac {1}{\sqrt{2}}(\pm 1, \pm 1, 0), \frac {1}{\sqrt{2}}(\pm 1, 0, \pm 1),
\frac {1}{\sqrt{2}}(0, \pm 1,  \pm 1).
\]
The unit sphere is smooth, so the constraint qualification
condition holds at every feasible point.
There is no inequality constraint, so strict complementarity is automatically satisfied.
It can be verified that the second order sufficiency condition \reff{sosc-optcnd}
holds on all the global minimizers. For instance, at $u=\frac {1}{\sqrt{3}}(1, 1, 1)$,
\[
\nabla_x^2L(u) = \frac 49
\left(
3\cdot \bbm 1 & 0 & 0 \\ 0 & 1 & 0 \\ 0 & 0 & 1 \ebm -
\bbm  1 \\ 1 \\ 1 \ebm  \bbm  1 \\ 1 \\ 1 \ebm^T\right), \quad
G(u)^\perp  = \bbm  1 \\ 1 \\ 1 \ebm^\perp.
\]
Clearly, \reff{sosc-optcnd} is satisfied at $u$.
By Theorem~\ref{mthm:opc=>fcvg},
Lasserre's hierarchy for this problem has finite convergence.
A numerical experiment by {\tt GloptiPoly~3} \cite{GloPol3} verified that
$f_5 = f_{min}=0$, modulo computer round-off errors.
\qed
\end{exm}

In Theorem~\ref{mthm:opc=>fcvg}, none of the optimality conditions there
can be dropped. We show counterexamples as follows.

\begin{exm}
(a) Consider the optimization problem:
\[
\left\{\baray{rl}
\min & 3x_1+2x_2\\
s.t. &  x_1^2-x_2^2-(x_1^2+x_2^2)^2\geq 0, \,  x_1\geq 0.
\earay\right.
\]
It can be shown that the origin $0$ is the unique global minimizer.
The constraint qualification condition fails at $0$,
and the first order optimality condition \reff{opcd:grad} fails.
The feasible set has nonempty interior,
so the SOS program \reff{sos:Put} achieves its optimal value (cf.~\cite{Las01}).
Lasserre's hierarchy for this problem does not have finite convergence,
which is implied by Proposition~\ref{ficvg=>FOOC} in the below. \\
(b) Consider the optimization problem:
\[
\left\{\baray{rl}
\min & x_1x_2+x_1^3+x_2^3\\
s.t. &  x_1\geq 0, x_2 \geq 0, 1-x_1-x_2 \geq 0.
\earay\right.
\]
Clearly, $0$ is the unique global minimizer.
The constraint qualification condition holds at $0$.
The Lagrange multipliers are all zeros.
The second order sufficiency condition \reff{sosc-optcnd}
also holds at $0$ because the null space $G(0)^\perp=\{0\}$.
However, the strict complementarity condition fails at $0$.
Lasserre's hierarchy for this problem, does not have finite convergence,
as shown by Scheiderer \cite[Remark~3.9]{Sch05JoA}. \\
(c) Consider the optimization problem:
\[
\left\{\baray{rl}
\min & x_1^4x_2^2+x_1^2x_2^4+x_3^6-3x_1^2x_2^2x_3^2+\eps (x_1^2+x_2^2+x_3^2)^3\\
s.t. & 1-x_1^2-x_2^2-x_3^2 \geq 0.
\earay\right.
\]
For every $\eps>0$, $0$ is the unique global minimizer,
and the constraint qualification and strict complementarity conditions hold at $0$.
However, the second order sufficiency condition fails at $0$.
For $\eps>0$ sufficiently small, Lasserre's hierarchy
for this optimization problem does not have finite convergence,
as shown by Marshall \cite[Example~2.4]{Mar06}.
\qed
\end{exm}

The first order optimality condition \reff{opcd:grad} is necessary
for Lasserre's hierarchy to have finite convergence.
This is summarized as follows.

\begin{prop} \label{ficvg=>FOOC}
Suppose \reff{sos:Put} achieves its optimal value.
If the first order optimality condition \reff{opcd:grad} fails
at a global minimizer of \reff{pop:gen},
then Lasserre's hierarchy cannot have finite convergence.
\end{prop}
\begin{proof}
Suppose otherwise $f_k = f_{min}$ for some $k$.
Since \reff{sos:Put} achieves its optimum,
\[
f - f_{min} =
\overset{m_1}{ \underset{ i=1}{\sum} }  \phi_i h_i
+ \overset{m_2}{ \underset{ j=0}{\sum} }   \sig_j g_j
\]
for some $\phi_i \in \re[x]$ and $\sig_j \in \Sig\re[x]^2$.
Let $u$ be a global minimizer of \reff{pop:gen}.
Note that every $h_i(u)=0$ and $g_j(u)\sig_j(u)=0$.
Differentiate the above with respect to $x$
and evaluate it at $u$, then we get
\[
\nabla f(u) =
\overset{m_1}{ \underset{ i=1}{\sum} }  \phi_i(u) \nabla h_i(u)
+ \overset{m_2}{ \underset{ j=0}{\sum} } (\sig_j(u) \nabla g_j(u) + g_j(u) \nabla \sig_j(u) ).
\]
Since every $\sig_j$ is SOS, $g_j(u)\sig_j(u)=0$ implies $g_j(u) \nabla \sig_j(u)=0$.
Hence,
\[
\nabla f(u) =
\overset{m_1}{ \underset{ i=1}{\sum} }  \phi_i(u) \nabla h_i(u)
+ \overset{m_2}{ \underset{ j=0}{\sum} } \sig_j(u) \nabla g_j(u).
\]
But this means that \reff{opcd:grad} holds at $u$,
which is a contradiction. So Lasserre's hierarchy
cannot have finite convergence.
\end{proof}

In Proposition~\ref{ficvg=>FOOC}, the assumption that \reff{sos:Put} achieves its optimal value
cannot be dropped. (This assumption is satisfied
if $K$ has nonempty interior, cf. \cite{Las01}.)
As a counterexample, consider the simple problem
\[
\min \quad x \qquad s.t. \quad -x^2 \geq 0.
\]
The global minimizer is $0$. The first order optimality condition fails at $0$,
but Lasserre's hierarchy has finite convergence
($f_k=f_{min}=0$ for all $k\geq 1$).

\section{Zariski Openness of Optimality Conditions}
\setcounter{equation}{0}

This section is mostly to prove Theroem~\ref{mthm:opcd=gen}.
For this purpose, we need some results on generic properties of critical points.

\subsection{Generic properties of critical points}

Given polynomials $p_0 \in \re[x]_{d_0}, \ldots, p_k \in \re[x]_{d_k}$ with $k\leq n$,
consider the optimization problem
\be \label{opt:p0pk}
\min_{x\in \re^n} \quad p_0(x) \quad s.t. \quad p_1(x) = \cdots = p_k(x) = 0.
\ee
Its Karush-Kuhn-Tucker (KKT) system is defined by the equations
\be  \label{eq:kkt-p01k}
\nabla_x p_0(x) - \sum_{i=1}^k \lmd_i \nabla_x p_i(x) = 0,
\quad p_1(x) = \cdots = p_k(x) = 0.
\ee
Every $(x,\lmd)$ satisfying \reff{eq:kkt-p01k} is called a {\it critical pair},
and such $x$ is called a {\it critical point}. Let
\be \label{ktRk:p=0}
\mc{K}(p) :=
\left\{ x \in \cpx^n
\left| \baray{c}
\rank \bbm \nabla_x p_0(x)& \nabla_x p_1(x) & \cdots & \nabla_x p_k(x) \ebm \leq k \\
\quad p_1(x) = \cdots = p_k(x) =0
\earay \right. \right\}
\ee
be the {\it KKT variety} of \reff{opt:p0pk}.
Clearly, every critical point belongs to $\mc{K}(p)$.

First, we discuss when does $\mc{K}(p)$ intersect the variety $q(x)= 0$
of a polynomial $q \in \re[x]_{d_{k+1}}$, i.e., when does the polynomial system
\be \label{kt-p:cap:q=0}
\left\{ \baray{c}
\rank \bbm \nabla_{x} p_0(x)& \nabla_{x} p_1(x) & \cdots & \nabla_{x} p_k(x) \ebm \leq k  \\
\quad p_1(x) = \cdots = p_k(x) =0, \quad q(x) =0
\earay \right.
\ee
have a solution in $\cpx^n$?
For a generic $p$, $\mc{K}(p)$ is a finite set (cf. \cite[Prop.~2.1]{NR09}),
and it does not intersect $q(x)=0$ if $q$ is also generic.
Consider the homogenization in $x:=(x_1,\ldots,x_n)$
of the polynomial system \reff{kt-p:cap:q=0}:
\be \label{hmg:ktp:q=0}
\left\{   \baray{c}
\rank \bbm \nabla_x \widetilde{p_0}(\tilde{x})  &
\nabla_x \widetilde{p_1}(\tilde{x}) & \cdots &
\nabla_x \widetilde{p_k}(\tilde{x}) \ebm \leq k,  \\
\quad \widetilde{p_1}(\tilde{x})  = \cdots = \widetilde{p_k}(\tilde{x})
= \widetilde{q}(\tilde{x}) = 0.
\earay \right.
\ee
Its variable is $\tilde{x}:=(x_0,\ldots,x_n)$.
When $k<n$, the matrix in \reff{hmg:ktp:q=0} has rank $\leq k$
if and only if all its maximal minors vanish;
when $k=n$, the rank condition in \reff{hmg:ktp:q=0} is always satisfied
and can be dropped.
Thus, in either case, \reff{hmg:ktp:q=0} can be equivalently defined by
some homogeneous polynomial equations, say,
\[
M_1(\tilde{x}) = \cdots = M_\ell(\tilde{x}) =0.
\]
Note that the coefficients of every $M_i$ are also homogeneous in
the ones of each of $p_0,\ldots, p_k, q$. By Theorem~\ref{thm:elim},
there exist polynomials
\[
R_1(p_0,\ldots,p_k;q), \ldots, R_t(p_0,\ldots,p_k;q)
\]
in the coefficients of $p_0,\ldots,p_k,q$ such that
\bit
\item every $R_i(p_0,\ldots,p_k;q)$ has integer coefficients
and is homogeneous in the coefficients of each of $p_0,p_1,\ldots, p_k,q$;

\item the system \reff{hmg:ktp:q=0}
has a solution $0 \ne \tilde{x} \in \cpx^{n+1}$ if and only if
\[
R_1(p_0,\ldots,p_k;q) = \cdots = R_t(p_0,\ldots,p_k;q) = 0.
\]
\eit
Define the polynomial $\mathscr{R}(p_0,\ldots,p_k;q)$ as
\be \label{ktRes:R(pq)}
\mathscr{R}(p_0,\ldots,p_k;q) :=
R_1(p_0,\ldots,p_k;q)^2+\cdots+R_t(p_0,\ldots,p_k;q)^2.
\ee
Note that $\mathscr{R}(p_0,\ldots,p_k;q)$ is a polynomial
in the coefficients of the tuple
\[
(p_0, \ldots, p_k, q) \in
\re[x]_{d_0} \times \cdots \times \re[x]_{d_k} \times \re[x]_{d_{k+1}}.
\]
Combining the above, we can get the following proposition.

\begin{prop} \label{pro:kkt-res}
Let $p_0 \in \re[x]_{d_0}, \ldots, p_k \in \re[x]_{d_k},q\in \re[x]_{d_{k+1}}$,
and $\mathscr{R}$ be as defined in \reff{ktRes:R(pq)}.
Then \reff{hmg:ktp:q=0} has a solution $0 \ne \tilde{x} \in \cpx^{n+1}$
if and only if $\mathscr{R}(p_0,\ldots,p_k;q) = 0$.
In particular, if $\mathscr{R}(p_0,\ldots,p_k;q) \ne 0$,
then \reff{kt-p:cap:q=0} has no solution in $\cpx^n$.
\end{prop}

We would like to remark that the polynomial $\mathscr{R}$ in \reff{ktRes:R(pq)}
does not vanish identically in
$
(p_0, \ldots, p_k, q) \in
\re[x]_{d_0} \times \cdots \times \re[x]_{d_k} \times \re[x]_{d_{k+1}},
$
for any given positive degrees $d_0, \ldots, d_k, d_{k+1}$.
A proof for this fact is given in the Appendix.

\bigskip

Second, we discuss when the KKT system \reff{eq:kkt-p01k} is nonsingular. Denote
\[
L_p(x,\lmd) := p_0(x) - \sum_{i=1}^k \lmd_i p_i(x).
\]
The polynomial system \reff{eq:kkt-p01k} is nonsingular
if and only if the square matrix
\[
H_{p}(x,\lmd) := \bbm \nabla_{x}^2 L_p(x,\lmd) &  Jac(p_1, \ldots, p_k)|_x^T   \\
Jac(p_1, \ldots, p_k)|_x  &  0 \ebm
\]
is nonsingular at every critical pair $(x,\lmd)$.
If every $p_i$ is generic, there are only finitely many critical pairs,
and \reff{eq:kkt-p01k} is nonsingular if $\det H_{p}(x,\lmd)$
does not vanish on them.

The matrix $H_{p}(x,\lmd)$ is singular if and only if there exists
$ (0,0) \ne (y, \nu) \in \re^n \times \re^k$ such that
\be \label{Hnull:(y,nu)}
\nabla_{x}^2 L_p(x,\lmd)  y + Jac(p_1, \ldots, p_k)|_{x}^T \nu = 0,
\quad y \in \bigcap_{i=1}^k \nabla p_i(x)^\perp.
\ee
When $Jac(p_1, \ldots, p_k)|_{x}$ has full rank $k$,
the existence of a pair $(y,\nu) \ne (0,0)$ satisfying \reff{Hnull:(y,nu)}
is equivalent to the existence of a pair $(y,\nu)$
with $y\ne 0$ satisfying \reff{Hnull:(y,nu)}.
When \reff{eq:kkt-p01k} is nonsingular,
there is no $y \ne 0$ satisfying \reff{Hnull:(y,nu)}
for any critical pair $(x,\lmd)$.
Write $\nu = (\nu_1,\ldots,\nu_k)$, then
\reff{eq:kkt-p01k} and \reff{Hnull:(y,nu)} together are equivalent to
\be  \label{gradHes=0:p}
\left\{\baray{c}
\bbm \nabla_x p_0 \\  (\nabla_x^2 p_0) y \ebm -
\overset{k}{\underset{i=1}{\sum}} \lmd_i \bbm \nabla_x p_i \\  (\nabla_x^2 p_i) y  \ebm +
\overset{k}{\underset{i=1}{\sum}} \nu_i \bbm  0 \\  \nabla_x p_i \ebm   = 0, \\
p_1(x) = \cdots = p_k(x) = (\nabla_x p_1)^T y = \cdots = (\nabla_x p_k)^T y =0.
\earay\right.
\ee
Define the $(2n)\times (2k+1)$ matrix
\[
P(x,y) : =
\bbm  \nabla_x p_0 \quad \cdots \quad \nabla_x p_k  &   \qquad \qquad \\
(\nabla_{x}^2 p_0)y  \quad \cdots \quad (\nabla_{x}^2 p_k)y &
\nabla_x p_1 \quad \cdots \quad \nabla_x p_k  \ebm.
\]
Clearly, every pair $(x,y)$ in \reff{gradHes=0:p} satisfies
\be  \label{rP<=2k:p^Ty=0}
\left\{\baray{c}
\rank \, P(x,y) \leq 2k,
\quad p_1(x) = \cdots = p_k(x) =  0, \\
(\nabla_x p_1)^T y = \cdots = (\nabla_x p_k)^T y =0.
\earay \right.
\ee
If the vectors
\[
\bbm \nabla_x p_i \\  (\nabla_x^2 p_i) y  \ebm,
\bbm  0 \\  \nabla_x p_i \ebm \quad(i=1,\ldots,k)
\]
are linearly independent,
\reff{gradHes=0:p} and \reff{rP<=2k:p^Ty=0} are equivalent.
Consider the homogenization in $x$ of \reff{rP<=2k:p^Ty=0}:
\be \label{ktHs-hg:xy-p}
\left\{\baray{c}
\rank \, \widetilde{P}(\tilde{x},y) \leq 2k, \quad
\widetilde{p_1}(\tilde{x}) = \cdots = \widetilde{p_k}(\tilde{x}) = 0, \\
(\nabla_x \widetilde{p_1})^Ty = \cdots = (\nabla_x \widetilde{p_k})^Ty = 0.
\earay\right.
\ee
In the above, $\tilde{x}:=(x_0,\ldots,x_n)$ and
\[
\widetilde{P}(\tilde{x},y) :=
\bbm
\nabla_x \widetilde{p_0}  \quad \cdots
\quad \nabla_x \widetilde{p_k}  &  \qquad \qquad \\
\bmx \big(\nabla_x^2 \widetilde{p_0}\big) y  & \cdots &
\big(\nabla_x^2 \widetilde{p_k}\big) y  \emx  &
\nabla_x \widetilde{p_1} \quad \cdots \quad \nabla_x \widetilde{p_k}
\ebm.
\]
When $k=n$, we always have $\rank\,\widetilde{P}(\tilde{x},y) \leq 2k$
and the rank condition in \reff{ktHs-hg:xy-p} can be dropped.
When $k<n$, we can replace $\rank\,\widetilde{P}(\tilde{x},y) \leq 2k$
by the vanishing of all maximal minors of $\widetilde{P}(\tilde{x},y)$.
In either case, \reff{ktHs-hg:xy-p} could be equivalently defined by
some polynomial equations, say,
\[
N_1(\tilde{x}, y) = \cdots = N_r(\tilde{x}, y)=0.
\]
Note that all $N_1,\ldots, N_r$ are homogeneous in both $\tilde{x}$ and $y$,
and their coefficients are also homogeneous in the ones of each of $p_0,p_1,\ldots,p_k$.
By applying Theorem~\ref{thm:elim} twice (first in $\tilde{x}$ and then in $y$),
there exist polynomials $D_i(p_0,p_1,\ldots, p_k)$ $(i = 1,\ldots, s)$,
in the coefficients of $p_0,p_1,\ldots,p_k$, such that
\bit
\item every $D_i(p_0,p_1,\ldots, p_k)$ has integer coefficients
and is homogeneous in the coefficients of
each of $p_0,p_1,\ldots, p_k$;

\item there exist $0 \ne \tilde{x} \in \cpx^{n+1}$ and $0 \ne y \in \cpx^n$
satisfying \reff{ktHs-hg:xy-p} if and only if
\[
D_1(p_0,p_1,\ldots, p_k) = \cdots = D_s(p_0,p_1,\ldots, p_k) = 0.
\]
\eit
Define the polynomial $\mathscr{D}(p_0,p_1,\ldots, p_k)$ as
\be \label{ktDis:D(p)}
\mathscr{D}(p_0,p_1,\ldots, p_k) := D_1(p_0,p_1,\ldots, p_k)^2+
\cdots+D_s(p_0,p_1,\ldots, p_k)^2.
\ee
Note that $\mathscr{D}(p_0,\ldots,p_k)$ is a polynomial
in the coefficients of the tuple
\[(
p_0, \ldots, p_k) \in
\re[x]_{d_0} \times \cdots \times \re[x]_{d_k}.
\]
Combining the above, we can get the following proposition.

\begin{prop} \label{po:kktDis}
Let $p_0 \in \re[x]_{d_0}, \ldots, p_k \in \re[x]_{d_k}$
and $\mathscr{D}$ be as defined in \reff{ktDis:D(p)}.
Then \reff{ktHs-hg:xy-p} has a solution
$(\tilde{x}, y) \in \cpx^{n+1} \times \cpx^n$ with $\tilde{x}\ne 0, y\ne 0$
if and only if $\mathscr{D}(p_0,\ldots,p_k) = 0$.
In particular, if $\mathscr{D}(p_0,\ldots,p_k) \ne 0$,
then \reff{eq:kkt-p01k} is a nonsingular system.
\end{prop}

The following special cases are useful to
illustrate Proposition~\ref{po:kktDis}.
\bit

\item  (Every $\deg(p_i)=1$.)
Let $p_i=a_i^Tx+b_i$ for $i=0,\ldots,k$.
If $k<n$ and $a_0, a_1,\ldots,a_k$ are linearly independent,
then $\widetilde{P}$ is a constant matrix of rank $2k+1$.
If $k=n$ and $a_1,\ldots,a_n$ are linearly independent,
then there is no $y\ne 0$ satisfying $\nabla_x \widetilde{p_i}^Ty=0$ for $i=1,\ldots,n$.
So, if every $p_i$ is generic, then \reff{ktHs-hg:xy-p}
has no complex solution $(\tilde{x}, y)$ with $\tilde{x}\ne 0, y\ne 0$.

\item  ($k=0$, i.e., \reff{opt:p0pk} has no constraints.)
The system \reff{ktHs-hg:xy-p} is then reduced to
\be \label{grad:Hes=0}
\nabla_x  \widetilde{p_0}(\tilde{x}) = 0, \quad
\left( \nabla_x^2  \widetilde{p_0}(\tilde{x}) \right) y =0.
\ee
If $\deg(p_0)=1$ and $p_0$ is nonzero, $\nabla_x  \widetilde{p_0}(\tilde{x}) = 0$
has no complex solution. If $\deg(p_0)=2$ and
$p_0=x^TAx+2b^Tx+c$ with $\det(A) \ne 0$, there is no $y\ne 0$ satisfying
$\left( \nabla_x^2  \widetilde{p_0}(\tilde{x}) \right) y =0$.
When $\deg(p_0)\geq 3$, by the definition of discriminants
for several polynomials (cf.~\S2.4),
\reff{grad:Hes=0} has a complex solution $(\tilde{x},y)$
with $\tilde{x}\ne 0,y \ne 0$ if and only if
\[
\Delta\left(\frac{\pt \widetilde{p_0}}{\pt x_1}, \ldots,
\frac{\pt \widetilde{p_0}}{\pt x_n}\right) = 0.
\]
So, if $p_0$ is generic, there are no
$\tilde{x}\ne 0,y \ne 0$ satisfying \reff{grad:Hes=0}.

\eit
The above observations can be simply implied by Proposition~\ref{po:kktDis}.

In Proposition~\ref{po:kktDis}, one might naturally think of
replacing $\mathscr{D}$  by
\be \label{dis-kkt}
\Delta( \nabla_x p_0 - Jac(p_1, \ldots, p_k)|_x^T \lmd, p_1, \ldots, p_k),
\ee
which is the discriminant for the set of polynomials defining \reff{eq:kkt-p01k},
by considering $\lmd_1,\ldots,\lmd_k$ as new variables, in addition to $x$.
However, this approach is problematic.
The main issue is that the discriminantal polynomial in \reff{dis-kkt}
might be identically zero, e.g., when $\deg(p_0) \leq \max_{1\leq i\leq k} \deg(p_i)$.
For convenience, consider the simple case $n>k=1$ and $a:=\deg(p_1) - \deg(p_0)\geq 0$.
By definition of discriminants for several polynomials (cf. \S2.4),
the discriminant in \reff{dis-kkt} vanishes if there exists
a complex vector $(x_0, x_1, \ldots, x_n, \lmd_1) \ne 0$ satisfying
\be \label{eq:dis=0}
\left\{\baray{c}
x_0^{a+1} \cdot \nabla_x \widetilde{p_0}  -
\lmd_1 \nabla_x \widetilde{p_1} =0 ,
\quad \widetilde{p_1}(x_0,\ldots,x_n)=0, \\
\det \bbm  x_0^{a+1} \cdot \nabla_{x}^2 \widetilde{p_0} -
\lmd_1 \nabla_x^2 \widetilde{p_1} &
\nabla_{x} \widetilde{p_1}   \\
 \nabla_{x} \widetilde{p_1}^T  &  0 \ebm
=0.
\earay \right.
\ee
Let $(u_1,\ldots,u_n) \ne 0$ be a complex zero of $\widetilde{p_1}(0,x_1,\ldots,x_n)$.
Then, $(0,u_1,\ldots,u_n, 0)$ is a nonzero solution of \reff{eq:dis=0}.
So, for any $p_0, p_1$, \reff{eq:dis=0} always has a nonzero
complex solution like $(0,u_1,\ldots,u_n, 0)$.
This means that the discriminant in \reff{dis-kkt} identically vanishes.
On the other hand, the polynomial $\mathscr{D}$ in \reff{ktDis:D(p)}
does not vanish identically in
$
(p_0, \ldots, p_k) \in
\re[x]_{d_0} \times \cdots \times \re[x]_{d_k},
$
for any given positive degrees $d_0, \ldots, d_k$.
A proof for this fact is given in the Appendix.

Typically, the polynomials
$\mathscr{R}$ in \reff{ktRes:R(pq)} and $\mathscr{D}$ in \reff{ktDis:D(p)}
are very difficult to compute explicitly.
They are mostly for theoretical interests.

\subsection{Zariski openness of optimality conditions}

This section is to prove that the constraint qualification,
strict complementarity and second order sufficiency conditions
all hold at every local minimizer of \reff{pop:gen}
if a finite set of polynomials, which are in the coefficients of polynomials
$f, h_i\,(i\in [m_1]), g_j \,(j\in [m_2])$,
do not not vanish at the input polynomials.
(That is, they hold in a Zariski open set in the space 
of input polynomials.)
These polynomials are listed as follows.

\begin{cond}  \label{cond:gen}
The polynomials $f_0 \in \re[x]_{d_0}$,
$h_i \in \re[x]_{d_i}$ ($i\in[m_1]$),
and $g_j\in \re[x]_{d_j^{\prm}}$ ($j\in[m_2]$)
with $m_1\leq n$ satisfy ($Res, \Delta$ are from \S2.4,
$\mathscr{R}$ from \reff{ktRes:R(pq)} and $\mathscr{D}$ from \reff{ktDis:D(p)}):

\bit

\item [(a)] If $m_1+m_2\geq n+1$, for all $1\leq j_1 < \cdots < j_{n-m_1+1} \leq m_2$,
\[
Res(h_1,\ldots,h_{m_1}, g_{j_1},\ldots,g_{j_{n-m_1+1}}) \ne 0.
\]

\item [(b)] For all $1\leq j_1<\cdots < j_r \leq m_2$ with $0\leq r \leq n-m_1$,
\[
\Delta(h_1,\ldots,h_{m_1}, g_{j_1},\ldots,g_{j_r}) \ne 0.
\]

\item [(c)] For all $1\leq j_1<\cdots < j_r \leq m_2$ with $0\leq r \leq n-m_1$,
\[
\mathscr{R}(f,p_1,\ldots, p_k; p_{k+1}) \ne 0,
\]
where $(p_1,\ldots, p_k, p_{k+1})$ is a re-ordering of
$(h_1,\ldots,h_{m_1}, g_{j_1},\ldots,g_{j_r})$.

\item [(d)] For all $1\leq j_1<\cdots < j_r \leq m_2$ with $0\leq r \leq n-m_1$,
\[
\mathscr{D}(f,h_1,\ldots,h_{m_1}, g_{j_1},\ldots,g_{j_r}) \ne 0.
\]

\eit

\end{cond}

First, we study the relationship between Condition~\ref{cond:gen}
and properties of critical points.
Let $u \in K$ be a critical point of \reff{pop:gen} (i.e.,
\reff{opcd:grad} and \reff{op-cd:comc} are satisfied for some $\lmd_i, \mu_j$,
excluding the sign conditions $\mu_j \geq 0$).
Let $J(u):=\{j_1,\ldots,j_r\}$ be the index set of active inequality constraints. Denote
\[
L(x) := f(x) - \sum_{i=1}^{m_1} \lmd_i h_i(x) - \sum_{ j \in J(u) } \mu_j g_j(x),
\]
\[
G(x) := \bbm \nabla h_1(x) & \cdots & \nabla h_{m_1}(x) &
\nabla g_{j_1}(x) & \cdots & \nabla g_{j_r}(x) \ebm^T,
\]
\[
H(x):= \bbm  \nabla_{x}^2 L(x)  & G(x)^T \\ G(x) & 0 \ebm.
\]

\begin{prop} \label{inq=>kktcd}
Let $u \in K$ and $\lmd_i, \mu_j$ satisfy
\reff{opcd:grad}-\reff{op-cd:comc} (excluding the sign conditions $\mu_j \geq 0$),
and $L(x), G(x), H(x)$ be as above.
Condition~\ref{cond:gen} has the following properties:
\bit

\item [i)] Item (a) implies that at most $n-m_1$ of $g_j$'s are active at every point of $K$.

\item [ii)] Item (b) implies that the constraint qualification condition
holds at every point of $K$.

\item [iii)] Item (c) implies that $\lmd_i \ne 0, \mu_j \ne 0$
for all $i\in [m_1]$ and $j \in J(u)$.

\item [iv)] Item (d) implies that $H(u)$ is nonsingular, i.e., $\det H(u) \ne 0$.
\eit
\end{prop}
\begin{proof}

i) If more than $n-m_1$ of $g_j$'s vanish at a point $u \in K$,
say, $g_{j_1},\ldots,g_{j_{n-m_1+1}}$,
then there are $n+1$ polynomials vanishing at $u$, including $h_1,\ldots,h_{m_1}$.
This implies the resultant
\[
Res(h_1,\ldots,h_{m_1}, g_{j_1},\ldots,g_{j_{n-m_1+1}}) =0,
\]
which violates item (a) of Condition~\ref{cond:gen}.
So, the item i) is true.

ii)
By item (b) of Condition~\ref{cond:gen},
$\Delta(h_1,\ldots, h_{m_1}, g_{j_1},\ldots, g_{j_r}) \ne 0$.
By the definition of $\Delta$ (cf. \S2.4),
the gradients of $h_1,\ldots, h_{m_1}, g_{j_1},\ldots, g_{j_r}$ at $u$
are linearly independent, i.e.,
the constraint qualification condition holds at $u$.

iii) Suppose otherwise one of $\lmd_i (i \in [m_1])$
or $\mu_j (j \in J(u))$ is zero, say, $\mu_{j_r} = 0$,
then $u$ is also a critical point of the optimization problem
\[
\min \quad f(x) \quad s.t. \quad h_i(x) = 0\, (i\in [m_1]), \,\,
 g_j(x) = 0 \,(j \in J(u)/\{j_r\}).
\]
Note that $g_{j_r}(u)=0$. By definition of $\mathscr{R}$ in \reff{ktRes:R(pq)}
and Proposition~\ref{pro:kkt-res}, we get
\[
\mathscr{R}(f,h_1,\ldots,h_{m_1}, g_{j_1},\ldots,g_{j_{r-1}}; g_{j_r})=0,
\]
which contradicts item (c) of Condition~\ref{cond:gen}.
So, the item iii) must be true.

iv) This is implied by definition of $\mathscr{D}$ in \reff{ktDis:D(p)}
and Proposition~\ref{po:kktDis}.
\end{proof}

Second, we study the relationship between Condition~\ref{cond:gen}
and the optimality conditions. This is summarized as follows.

\begin{prop} \label{opcd=ZarOpen}
If Condition~\ref{cond:gen} holds, then the constraint qualification,
strict complementarity and second order sufficiency conditions
all hold at every local minimizer of \reff{pop:gen}.
This is implied by the following properties:

\bit

\item [1)] Item (a) of Condition~\ref{cond:gen} implies that
at most $n-m_1$ of $g_j$'s are active
at every local minimizer of \reff{pop:gen}.

\item [2)] Item (b) of Condition~\ref{cond:gen} implies that
the constraint qualification condition holds
at every local minimizer of \reff{pop:gen}.

\item [3)] Items (b) and (c) of Condition~\ref{cond:gen} imply that
the strict complementarity condition
holds at every local minimizer of \reff{pop:gen}.

\item [4)] Items (b) and (d) of Condition~\ref{cond:gen} imply that
the second order sufficiency condition
holds at every local minimizer of \reff{pop:gen}.

\eit
\end{prop}
\begin{proof}
Let $u$ be a local minimizer of \reff{pop:gen}.

1) and 2) are implied by i), ii) of Proposition~\ref{inq=>kktcd} respectively.

3) By item 2), the constraint qualification condition holds at $u$.
So, there exist $\lmd_i, \mu_j$ satisfying \reff{opcd:grad}-\reff{op-cd:comc}
with all $\mu_j \geq 0$.
If $j \not\in J(u)$, then $g_j(u)>0$ and $\mu_j + g_j(u)>0$;
if $j \in J(u)$, then $\mu_j \ne 0$ by item iii) of Proposition~\ref{inq=>kktcd},
and hence $\mu_j>0$ and $\mu_j + g_j(u)>0$.
This means that the strict complementarity condition holds at $u$.

4) By item 2), the constraint qualification condition holds at $u$.
So, \reff{opcd:grad} and \reff{op-cd:comc} are satisfied.
The second order sufficiency condition is then implied by item iv) of Proposition~\ref{inq=>kktcd}
and Lemma~\ref{lm:SOSC=nsig} in the below.
\end{proof}

\begin{lemma} \label{lm:SOSC=nsig}
Let $u$ be a local minimizer of \reff{pop:gen}, $\lmd_i, \mu_j$ satisfy
\reff{opcd:grad}-\reff{op-cd:comc},
and $L(x), G(x), H(x)$ be as defined preceding Proposition~\ref{inq=>kktcd}.
If $G(u)$ has full rank, then \reff{sosc-optcnd}
holds at $u$ if and only if $\det H(u) \ne 0$.
\end{lemma}
\begin{proof}
First, assume \reff{sosc-optcnd} holds.
Then, for $\eta > 0$ big enough,
\[
V := \nabla_{x}^2 L(u) + \eta G(u)^T G(u) \succ  0.
\]
By the matrix equation
\[
\bbm  I_{n}  & \half \eta G(u)^T \\ 0  &  I_{m_1+r} \ebm
\bbm  \nabla_{x}^2 L(u)  & G(u)^T \\ G(u) & 0 \ebm
\bbm  I_{n}  & 0 \\ \half \eta G(u)  &  I_{m_1+r} \ebm  =
\]
\[
\bbm  \nabla_{x}^2 L(u) + \eta G(u)^TG(u)  & G(u)^T \\ G(u) & 0 \ebm ,
\]
one can see that
\[
\det H(u) = \det (V) \cdot \det \left( - G(u) V^{-1} G(u)^T \right) \ne 0,
\]
because of the positive definiteness of $V$ and nonsingularity of $G(u)$.

Second, assume $\det H(u) \ne 0$.
Suppose otherwise \reff{sosc-optcnd} fails.
Then there exists $ 0 \ne v\in G(u)^\perp$ such that
$
v^T \nabla_{x}^2 L(0) v \leq 0.
$
Since $G(u)$ has full rank, the constraint qualification condition holds at $u$.
So, the second order necessity condition \reff{opt:sonc} is satisfied at $u$.
It implies that $v$ is a minimizer of the problem
\[
\min_{z\in \re^n} \quad z^T \Big( \nabla_{x} ^2 L(u) \Big)z \quad s.t. \quad G(u) z =0.
\]
By the first order optimality condition for the above,
there exists $\nu$ such that $\nabla_{xx}^2 L(u) v = G(u)^T \nu$,
which then implies
\[
\bbm  \nabla_{xx}^2 L(u)  & G(u)^T \\ G(u) & 0 \ebm  \bbm v \\ -\nu \ebm =0.
\]
This contradicts $\det H(u) \ne 0$, because $v \ne 0$.
So, \reff{sosc-optcnd} must hold at $u$.
\end{proof}

We conclude this section with the proof of Theorem~\ref{mthm:opcd=gen}.

\begin{proof}[Proof of Theorem~\ref{mthm:opcd=gen}]
Let $\varphi_1,\ldots,\varphi_L$ be the finite set of polynomials
given in Condition~\ref{cond:gen}. Theorem~\ref{mthm:opcd=gen}
is then implied by Proposition~\ref{opcd=ZarOpen}.
\end{proof}

\section{Some discussions}
\setcounter{equation}{0}

Our main conclusions are Theorems~\ref{mthm:opc=>fcvg} and \ref{mthm:opcd=gen}.
Lasserre's hierarchy has finite convergence when the constraint qualification,
strict complementarity and second order sufficiency conditions hold
at every global minimizer, under the archimedean condition.
These optimality conditions are all satisfied at every local minimizer
if the vector of coefficients of input polynomials
lies in a Zariski open set. This gives a connection between the classical
nonlinear programming theory and Lasserre's hierarchy of semidefinite relaxations
in polynomial optimization. These results give an interpretation
for the phenomenon that Lasserre's hierarchy often has finite convergence
in solving polynomial optimization problems.

Under the assumptions that Condition~\ref{cod:par-u}
holds at every $u\in K$ and $K$ is irreducible and bounded,
Marshall \cite[Corollary~4.5]{Mar09} proved that, for each $d\geq 2$, the set
\[
\big\{ f \in \re[x]_d: \, f \mbox{ satisfies BHC at each global minimizer on } K \big\}
\]
is open and dense in $\re[x]_d$. This interesting result can also be implied by
Theorems \ref{mthm:opcd=gen} and \ref{thm:opcd=>BHC}.
Indeed, they can imply the following stronger conclusions:
\bit
\item
the boundary hessian condition is satisfied in a Zariski open set
in the space of input data
(not every open dense set is Zariski open, e.g., $\re^n\backslash\Z^n$);

\item Condition~\ref{cod:par-u} also holds in a Zariski open set;

\item for the case $d=1$,
the boundary hessian condition also  holds in a Zariski open set;

\item
the defining polynomials for $K$ are also allowed to be generic;
the set $K$ is not required to be irreducible or bounded.
\eit

We would like to remark that Putinar's Positivstellensatz
(cf.~Theorem~\ref{Put-Pos}) also holds generically for polynomials
that are nonnegative on $K$. Assume $\langle h \rangle + Q(g)$ is archimedean
and the ideal $\langle h \rangle$ is real.
Let $P_d(K)$ be the cone of polynomials in $\re[x]_d$
that are nonnegative on $K$, and $\pt P_d(K)$ be the boundary of $P_d(K)$.
Theorems~\ref{mthm:opcd=gen}, \ref{thm:marBHC} and \ref{thm:opcd=>BHC}
imply that if $p$ lies generically on $\pt P_d(K)$
then $p \in \langle h \rangle + Q(g)$.
In \cite{Las01}, Lasserre interpreted Putinar's Positivstellensatz
as a generalized KKT condition for global optimality.
Therefore, the classical KKT conditions for local optimality
and the generalized KKT condition
(i.e., Putinar's Positivstellensatz, under archimedeanness)
for global optimality, both hold generically.

A theoretically interesting question is whether there is
a uniform bound on the number of steps to achieve finite convergence
for Lasserre's hierarchy in the generic case. That is,
whether there exists an integer $N$, which only depends on
the degree of $f$ and a set of defining polynomials for $K$, such that
$f_k = f_{min}$ for all generic $f$ of a given degree and $k\geq N$?
Unfortunately, such a bound $N$ typically does not exist.
This could be implied by a result of Scheiderer \cite{Sch05} on
the non-existence of degree bounds for weighted SOS representations.
For instance, when $K$ is the $3$-dimensional unit ball,
such a bound does not exist (cf. \cite[Section~5]{Nie-ft}).

The archimedean condition cannot be
removed in Theorems~\ref{mthm:opc=>fcvg}, \ref{thm:marBHC}.
For instance, consider the unconstrained optimization
\[
\min \quad x_1^2x_2^2(x_1^2+x_2^2-3x_3^2)+x_3^6 + x_1^2+x_2^2+x_3^2.
\]
The origin $0$ is the unique minimizer.
The archimedean condition failed, because the feasible set
is the entire space $\re^n$ and is not compact.
The objective $f$ is the sum of the Motzkin polynomial and
the positive definite quadratic form $x^Tx$.
The second order sufficiency condition hold at $0$.
However, for all scalar $\gamma$, $f-\gamma$ is not SOS.
In this case, Lasserre's hierarchy does not converge.

The archimedean condition is not generically satisfied.
To see this fact, consider the simple case that $m_1=0$ and $m_2=1$.
For a given $d$, let $\mathscr{A}(d)$
be the set of polynomials $g \in \re[x]_d$ such that $Q(g)$ is archimedean.
The set $\mathscr{A}(d)$ is not dense in $\re[x]_d$.
For instance, when $d=2$, both $\mathscr{A}(d)$
and its complement $\re[x]_d\backslash \mathscr{A}(d)$
have nonempty interior:
\bit

\item  Let $b_1 = x^Tx-1$. Clearly, $b_1 \not\in \re[x]_2 \backslash \mathscr{A}(2)$.
For all $q \in \re[x]_2$ with sufficiently small coefficients,
we have $b_1+q \not\in \mathscr{A}(2)$.

\item  Let $b_2 = 1-x^Tx$. Clearly, $b_2 \in \mathscr{A}(2)$.
For all $q \in \re[x]_2$ with small sufficiently coefficients,
we have $b_2+q \in \mathscr{A}(2)$.

\eit

There exist polynomial optimization problems
that Lasserre's hierarchy fails to have finite convergence,
e.g., minimizing the Motzkin polynomial over the unit ball.
Such problems always exist when the feasible set has dimension
three or higher, as shown by Scheiderer \cite{Sch09}.
So, we are also interested in methods that have finite convergence
for optimizing {\it all} polynomials over a given set $K$.
The Jacobian SDP relaxation
is a method that has this property (cf. \cite{Nie-jac}).

Theorems~\ref{mthm:opc=>fcvg} does not tell how to check
when finite convergence happens. This can be done
by using {\it flat truncation},
which is a rank condition on the dual optimizers of \reff{sos:Put}.
Flat truncation is a sufficient condition
for Lasserre's hierarchy to have finite convergence.
In the generic case, flat truncation is also a necessary condition
for Lasserre's hierarchy to have finite convergence (cf.~\cite{Nie-ft}).

%
%

No matter Lasserre's hierarchy has finite convergence or not,
if there are finitely many global minimizers and the archimedean condition holds,
then the flat truncation condition is always asymptotically satisfied
(cf.~\cite[\S3]{Nie-ft}). So, in numerical experiments,
we might also observe that Lasserre's hierarchy has finite convergence
even if it does not have in exact mathematical computations.
However, if there are infinitely many global minimizers,
the flat truncation condition is typically not satisfied
(cf.~Laurent \cite[\S6.6]{Lau}). For instance, consider the problem
\[
\left\{
\baray{rl}
\min & x_1^2x_2^2(x_1^2+x_2^2)+(x_3-1)^6-3x_1^2x_2^2(x_3-1)^2 \\
s.t. & x_1^2+x_2^2+x_3^2- \half \geq 0, 2-x_1^2-x_2^2-x_3^2 \geq 0.
\earay \right.
\]
The objective $f$ is shifted from the Motzkin polynomial
(i.e., $f(x_1,x_2,x_3+1)$ is the Motzkin polynomial).
It has infinitely many global minimizers and $f_{min}=0$.
Lasserre's hierarchy for this problem does not have finite convergence.
This can be implied by the proof of Prop. 6.1 of Scheiderer \cite{Sch99},
because $(0,0,1)$ is a zero $f$ lying in the interior of the feasible set
and $f$ is a nonnegative but non-SOS form in $(x_1,x_2,x_3-1)$.
The flat truncation condition is typically not satisfied
for dual optimizers of \reff{sos:Put}.
When {\tt GloptiPoly~3} is applied to solve this problem numerically,
the convergence did not occur for the orders $k=3,4,\ldots,12$.

\bigskip
\noindent
{\bf Acknowledgement} \,
The author was partially supported by NSF grants DMS-0757212 and DMS-0844775.
He would like very much to thank Murray Marshall
for communications on the boundary hessian condition.

\appendix
\section{Non-identically Vanishing of $\mathscr{R}$ and $\mathscr{D}$}
\setcounter{equation}{0}

Given any positive degrees $d_0, d_1, \ldots, d_k, d_{k+1}$, we show that
the polynomial $\mathscr{R}(p_0,\ldots,p_k;q)$ defined in \reff{ktRes:R(pq)}
and the polynomial $\mathscr{D}(p_0,p_1,\ldots,p_k)$
defined in \reff{ktDis:D(p)} do not vanish identically
in $p_i \in \re[x]_{d_i}\,(i=0,\ldots,k)$ and $q \in \re[x]_{d_{k+1}}$.
Without loss of generality, we can assume all $d_1, \ldots, d_k>1$
because linear constraints in \reff{eq:kkt-p01k} can be removed by eliminating variables.

\bigskip

First, we prove that the polynomial $\mathscr{R}$ defined in \reff{ktRes:R(pq)}
does not vanish identically in the space
$\re[x]_{d_0} \times \cdots \times \re[x]_{d_k} \times \re[x]_{d_{k+1}}$.
We only consider the case $k<n$, because if $k=n$
then $\widetilde{p_1}(\tilde{x})=\cdots=\widetilde{p_k}(\tilde{x})
=\widetilde{q}(\tilde{x})$ has no nonzero complex solution in the generic case.
By Proposition~\ref{pro:kkt-res},
it is enough to show that the homogeneous polynomial system \reff{hmg:ktp:q=0}
does not have a complex solution $\tilde{x} \ne 0$ for generic $p_0, p_1,\ldots, p_k, q$.
We prove this in two cases:
\bit

\item ($\mathbf{x_0 \ne 0}$)\,
We can scale as $x_0=1$, and
the system  \reff{hmg:ktp:q=0} is then reduced to \reff{kt-p:cap:q=0}.
When $p_0, p_1,\ldots, p_k$ are generic,
the set $\mc{K}(p)$ defined in \reff{ktRk:p=0} is finite (cf. \cite[Prop.~2.1]{NR09}).
Thus, when $q$ is also generic, \reff{kt-p:cap:q=0}
does not have a solution in $\cpx^n$.

\item ($\mathbf{x_0 = 0}$)\,  The system  \reff{hmg:ktp:q=0} is then reduced to
\be \label{ktp-q:x0=0}
\left\{   \baray{c}
\rank \bbm \nabla_x p_0^h(x)  & \nabla_x p_1^h (x) & \cdots &
\nabla_x p_k^h(x) \ebm \leq k,  \\
\quad p_1^h(x)  = \cdots = p_k^h(x) = q^h(x) = 0.
\earay \right.
\ee
(Here, $f^h$ denotes the homogeneous part of the highest degree
for a polynomial $f$.)
When $p_1,\ldots, p_k$ are generic, we have $\Delta(p_1^h,\ldots, p_k^h) \ne 0$.
By definition of $\Delta$ (cf. \S2.4),
if $p_1^h(x)  = \cdots = p_k^h(x) =0$ and $x \ne 0$, then
\[
\rank \bbm  \nabla_x p_1^h (x) & \cdots & \nabla_x p_k^h(x) \ebm = k.
\]
So, if $x$ satisfies \reff{ktp-q:x0=0},
there must exist scalars $c_1,\ldots, c_k$ such that
\[
\nabla_x p_0^h(x)  = c_1 \nabla_x p_1^h (x) + \cdots + c_k \nabla_x p_k^h(x).
\]
Since each $p_i^h$ is a form, by Euler's formula
for homogeneous polynomials (cf. \cite[\S2]{Nie-dis}),
we can get
\[
d_0 p_0^h(x) = x^T \nabla_x p_0^h(x)
=\sum_{i=1}^k c_i x^T \nabla_x p_k^h(x)
=\sum_{i=1}^k c_i d_k p_k^h(x) = 0.
\]
This means that \reff{ktp-q:x0=0} implies
\[
p_0^h(x) = p_1^h(x)  = \cdots = p_k^h(x) =0,
\]
\[
\rank \bbm \nabla_x p_0^h (x) &  \nabla_x p_1^h (x) &
\cdots & \nabla_x p_k^h(x) \ebm = k.
\]
Any $x$ satisfying the above must be zero if
$\Delta(p_0^h, p_1^h,\ldots,p_k^h) \ne 0$.

\eit
Combining the above two cases, we know the polynomial system \reff{hmg:ktp:q=0}
has no complex solution $\tilde{x} \ne 0$
when $p_0, p_1,\ldots, p_k,q$ are generic.

\bigskip

Second, we show that the polynomial $\mathscr{D}(p_0,p_1,\ldots,p_k)$
defined in \reff{ktDis:D(p)} does not identically vanish in the space
$\re[x]_{d_0} \times \re[x]_{d_1} \times \cdots \times \re[x]_{d_k}$.
By Proposition~\ref{po:kktDis}, it is enough to prove that
there exist $p_i \in \re[x]_{d_i} \,(i=0,\ldots,k)$ such that
\reff{ktHs-hg:xy-p} has no complex solution $(\tilde{x},y)$
with $\tilde{x}\ne 0, y\ne 0$. We prove this in two cases.

\bit

\item ($\mathbf{x_0 \ne 0}$)\,  We scale as $x_0 = 1$,
and \reff{ktHs-hg:xy-p} is then reduced to \reff{rP<=2k:p^Ty=0}.
Choose polynomials $\hat{p}_i$ as follows:
\[
\hat{p}_0:= f_0 \in \re[x_{k+1},\ldots,x_n]_{d_0}, \,
\hat{p}_1:=x_1^{d_1}-1, \ldots, \hat{p}_k:=x_k^{d_k}-1.
\]
Clearly, on the variety $V(\hat{p}_1,\ldots,\hat{p}_k)$,
the gradients $\nabla_x \hat{p}_1,\ldots, \nabla_x \hat{p}_k$
are linearly independent, and so are
\[
\bbm \nabla_x \hat{p}_i \\  (\nabla_x^2 \hat{p}_i) y  \ebm,
\bbm  0 \\  \nabla_x \hat{p}_i \ebm \quad(i=1,\ldots,k).
\]
Thus, \reff{rP<=2k:p^Ty=0} is equivalent to \reff{gradHes=0:p}.
If $(x,\lmd)$ is a critical pair, then $\lmd_1 = \cdots = \lmd_k =0$ and
$D:=\diag(d_1x_1^{d_1-1}, \ldots,d_kx_k^{d_k-1})$ is invertible.
Denote $x_{I}:=(x_1,\ldots,x_k,)$ and $x_{II}:=(x_{k+1},\ldots,x_n)$.
Note that ($\hat{p}:=(\hat{p}_0,\ldots,\hat{p}_k)$)
\[
H_{\hat{p}}(x,0) =
\left[\baray{c|c|c}
0 & 0 & D\\ \hline
0 & \nabla_{x_{II} }^2 f_0 &  0 \\ \hline
D  &  0  &   0
\earay \right].
\]
(In the above, the $0$'s denote zero matrices of proper dimensions.)
The matrix $H_{\hat{p}}(x,0)$ is nonsingular if and only if
$ \nabla_{x_{II} }^2 f_0 $ is nonsingular.
Therefore, \reff{gradHes=0:p} has a solution
if and only if there exists $u \in \cpx^{n-k}$ satisfying
\[
\nabla_{x_{II}} f_0(u) =0, \quad  \det \,
\nabla_{x_{II}}^2 f_0(u)  = 0.
\]
However, the above is possible only if
\[
\Delta(\frac{\pt f_0}{\pt x_{k+1}}, \ldots, \frac{\pt f_0}{\pt x_{n}}) = 0.
\]
So, if $f_0$ is generic, then $H_{\hat{p}}(x,0)$ is nonsingular
for all $(x,\lmd)$ satisfying \reff{eq:kkt-p01k} corresponding to
$\hat{p}_0, \hat{p}_1, \ldots, \hat{p}_k$.

By continuity of roots of polynomials, $H_{p}(x,\lmd)$ is nonsingular
for every pair $(x,\lmd)$ satisfying \reff{eq:kkt-p01k},
if each $p_i$ is generic and close enough to $\hat{p_i}$.

\item ($\mathbf{x_0 = 0}$)\,  The polynomial system \reff{ktHs-hg:xy-p}
is then reduced to
\be \label{Hes-sig:hmg}
\left\{\baray{c}
\rank \, Q(x,y) \leq 2k, \quad
p_1^h(x) = \cdots = p_k^h(x) = 0, \\
\big(\nabla_{x} p_1^h(x)\big)^Ty = \cdots = \big(\nabla_{x} p_1^h(x)\big)^Ty = 0.
\earay\right.
\ee
%
%
In the above, $Q(x,y)$ denotes the matrix
\[
\bbm
\nabla_{x} p_0^h  \quad \cdots
\quad \nabla_{x}  p_k^h  &  \qquad \qquad \\
\bmx \big(\nabla_{x}^2 p_0^h\big) y  & \cdots &
\big(\nabla_{x}^2 p_k^h\big) y  \emx  &
\nabla_{x} p_1^h \quad \cdots \quad \nabla_{x} p_k^h
\ebm.
\]
We show that if $p_0,p_1,\ldots,p_k$ are generic, then
\reff{Hes-sig:hmg} has no complex solution $(x,y)$
with $x\ne 0, y\ne 0$.
When all $p_i$ are generic, for every $x\ne 0$ satisfying
\[
p_1^h(x) = \cdots = p_k^h(x) = 0,
\]
the gradients
$\nabla_{x} p_1^h, \ldots, \nabla_{x}  p_k^h $ are linearly independent.
When $\rank \, Q(x,y) \leq 2k$,
there exist scalars $c_1, \ldots, c_k$ such that
\be  \label{eq:kkt-hmg}
\nabla_x p_0^h(x) - \sum_{i=1}^k c_i \nabla_x p_i^h(x) = 0.
\ee
By Euler's formula for homogeneous polynomials, the above implies
\[
d_0 p_0^h(x) = x^T \nabla_x p_0^h(x)  = \sum_{i=1}^k \lmd_i x^T \nabla_x p_i^h(x)
= \sum_{i=1}^k \lmd_i d_i p_i^h(x) = 0.
\]
This means that if some $x\ne 0$ satisfies \reff{Hes-sig:hmg}
then the polynomial system
\[
p_0^h(x) = p_1^h(x) = \cdots = p_k^h(x) = 0
\]
is singular.
But this is impossible unless $\Delta(p_0^h, p_1^h,\ldots, p_k^h) = 0$.

\eit
Combining the above two cases, we know that there exist polynomials
$p_i \in \re[x]_{d_i}\,(i=0,\ldots,k)$ such that
there are no complex $\tilde{x} \ne 0, y\ne 0$ satisfying
\reff{ktHs-hg:xy-p}.
This shows that $\mathscr{D}(p_0,\ldots,p_k)$
does not identically vanish.

\end{document}